\documentclass[11pt,psamsfonts]{amsart}
\usepackage{amsmath}
\usepackage{amsthm}
\usepackage{amssymb}
\usepackage{amscd}
\usepackage{amsfonts}
\usepackage{amsbsy}
\usepackage{epsfig,afterpage}
\usepackage[dvips]{psfrag}
\usepackage{color}     
\usepackage{overpic}

\newcommand{\R}{\ensuremath{\mathbb{R}}}

\newtheorem {theorem} {Theorem} 
\newtheorem {proposition} [theorem] {Proposition}

\newtheorem {lemma} [theorem] {Lemma}
\newtheorem {definition} [theorem] {Definition}
\newtheorem {remark} {Remark}

\begin{document}


\title[Birth of isolated nested cylinders and limit cycles in 3D PSVFs] {Birth of isolated nested cylinders and limit cycles in 3D piecewise smooth vector fields with symmetry}

\author[T. Carvalho and B.R. Freitas]
{Tiago Carvalho$^1$ and Bruno Rodrigues de Freitas$^2$}

\address{$^1$ Departamento de Matem\'{a}tica, Faculdade de Ci\^{e}ncias,
	UNESP, Av. Eng. Luiz Edmundo Carrijo Coube 14-01,  CEP 17033-360,
	Bauru, SP, Brazil.}\email{tcarvalho@fc.unesp.br}

\address{$^2$ Universidade Federal de Goi\'{a}s, IME, CEP 74001-970, Caixa Postal 131, Goi\^{a}nia, Goi\'{a}s, Brazil.}
\email{freitasmat@ufg.br}


\date{}
\dedicatory{}

\maketitle

\begin{abstract}
Our start point is a 3D piecewise smooth vector field defined in two zones and presenting a shared fold curve for the two smooth vector fields considered. Moreover, these smooth vector fields are symmetric relative to the fold curve, giving raise to a continuum of nested topological cylinders such that each orthogonal section of these cylinders is filled by centers. First we prove that the normal form considered represents a whole class of piecewise smooth vector fields. After we perturb the initial model in order to obtain exactly $\mathcal{L}$ invariant planes containing centers. A second perturbation of the initial model also is considered in order to obtain exactly $k$ isolated cylinders filled by periodic orbits. Finally, joining the two previous bifurcations we are able to exhibit a model, preserving the symmetry relative to the fold curve, and having exactly $k.\mathcal{L}$ limit cycles.
\end{abstract}

\section{Introduction}\label{secao introducao}

Vector fields tangent to foliations, Hamiltonian systems and first integrals of vector fields are correlated themes very exploit in the literature about Dynamical Systems. In fact the list of papers on these subjects is extremely large and we cite just the books \cite{Livro-Bluman,Livro-CamachoNet,Livro-LlibreMoeckelSimo,Livro-Hamiltonian} for a brief notion on these issues.

Many authors have used the theoretical aspects about vector fields tangent to foliations, Hamiltonian systems and first integrals of vector fields in order to obtain dynamical properties of models describing some system in applied science. A far from exhaustive list of books in this sense is given by \cite{Livro-Bogolyubov,Livro-Chrusciel,Livro-Greiner}.

In recent years, scientists are realizing the importance and applicability of a new branch of dynamical systems that are powerful tools in phenomena where some \textit{``on-off''} phenomena take place. For example, in control theory (see \cite{Rossa}), mechanics models (see \cite{Brogliato,Dixon,Leine}), electrical circuits (see \cite{Kousaka}), relay systems (see  \cite{diBern-relay,Jac-To}), biological models (with refuge see \cite{Krivan}, foraging predators see \cite{Piltz}), among others where an instantaneous change on the system is observed when any barrier is broken. These dynamical systems are modeled by \text{``pieces''} and are called \textbf{piecewise smooth vector fields} (PSVFs for short).

Many authors have contributed to provide a general and consistent theory about PSVFs. We cite here the works \cite{diBern-livro,Simpson} where a non familiar reader can found the main definitions, conventions and results on this theory. However, very little have been studied about PSVFs tangent to (piecewise) foliations, Hamiltonian PSVFs and first integrals of PSVFs. Addressing this topic we cite \cite{CarvalhoTeixeira-JDE2016,JinWuHuang,Livro-Luo}.

The present paper deals precisely with PSVFs tangent to piecewise foliations. We found first integrals for them and perform bifurcations on the unstable PSVFs obtained. In fact, a very rich behavior is observed and, which it is very important, an almost fully exploit study area is brought to the surface.

\subsection{Setting the problem and statement of the main results}\label{secao colocacao problema}

Let  $\Sigma$ be a codimension one 3D manifold   given by
$\Sigma =f^{-1}(0),$ where $f:\R^3 \rightarrow \R$ is a smooth function
having $0\in \R$ as a regular value (i.e. $\nabla f(p)\neq 0$, for
any $p\in f^{-1}({0}))$.  We call $\Sigma$ the \textit{switching
	manifold} that is the separating boundary of the regions
$\Sigma^+=\{q\in \R^3 \, | \, f(q) \geq 0\}$ and $\Sigma^-=\{q \in \R^3 \,
| \, f(q)\leq 0\}$. 

Take $X: \Sigma^+ \rightarrow \R^3$ (resp., $Y: \Sigma^- \rightarrow \R^3$) smooth vector fields. We combine them in order to constitute the PSVF $Z: \R^3 \rightarrow \R ^{3}$ given by
\[
Z(x,y,z)=\left\{\begin{array}{l} X(x,y,z),\quad $for$ \quad (x,y,z) \in
\Sigma^+,\\ Y(x,y,z),\quad $for$ \quad (x,y,z) \in \Sigma^-.
\end{array}\right.
\]
 The trajectories of $Z$ are
solutions of  ${\dot q}=Z(q)$ and we will accept that $Z$ is
multi-valued in points of $\Sigma$. The basic results of
differential equations, in this context, were stated in \cite{Fi}. We use the notation $Z=(X,Y)$. 

Given $p \in \Sigma$, throughout this paper we do not consider the situation where both vector fields $X$ and $Y$ have trajectories arriving (resp. departing) from $p$ transversally. In these cases $p$ is called in the literature as a \textbf{sliding} (resp. \textbf{escaping}) point. So, here we assume that when an $X$-trajectory reaches $p \in \Sigma$ transversally, then there is a $Y$-trajectory starting at $p$ and transversal to $\Sigma$, i.e., generically, just \textbf{crossing} points will be considered.

In fact, the initial model that we consider is
\begin{equation}\label{eq forma normal centro}
Z_0(x,y,z)= \left\{
\begin{array}{lll}
X_0(x,y,z) = \left(
\begin{array}{c}
0 \\
-1  \\
2 y
\end{array}
\right)
& \hbox{if $z \geq 0$,} \\
Y_0(x,y,z) =  \left(
\begin{array}{c}
0 \\
1  \\
2 y
\end{array}
\right)& \hbox{if $z \leq 0$.}
\end{array}
\right.
\end{equation}The phase portrait of \eqref{eq forma normal centro} is given in Figure \ref{cilindro invariante}.
\begin{figure}[h!]
	\begin{center}
		\begin{overpic}[width=2.5in]{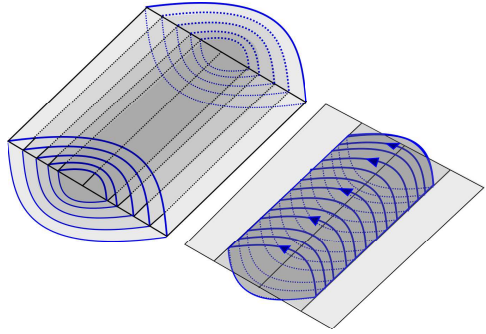}
		\end{overpic}
	\end{center}
	\caption{Topological cylinders.}\label{cilindro invariante}
\end{figure}

It is patent the symmetry of the trajectories obtained from \eqref{eq forma normal centro}. Moreover, we get that $H_1(x,y,z)=x$ and $H_2(x,y,z)=z+y^2$ (resp., $L_1(x,y,z)=x$ and $L_2(x,y,z)=z-y^2$) are independent first integrals of $X_0$ (resp., $Y_0$). The orbits of $X_0$ are contained in the sets $\{H_1=c_1\} \cap \{H_2=c_2\} \cap \{z\geq0\}$ and the orbits of $Y_0$ are contained in the sets $\{L_1=c_3\} \cap \{L_2=c_4\} \cap \{z\leq0\}$, with $c_1,c_2,c_3,c_4\in\R$.

So, a pair of \textit{piecewise first integrals} of \eqref{eq forma normal centro} is
\[
M_1(x,y,z)= x \mbox{ and } M_2(x,y,z)=\left\{
\begin{array}{lll}
H_2(x,y,z)
& \hbox{if $z \geq 0$,} \\
L_2(x,y,z) & \hbox{if $z \leq 0$.}
\end{array}
\right.
\]
Of course, the trajectories of $X_0$ (resp., $Y_0$) leave at the intersection of the transversal (in fact, orthogonal) foliations $H_1$ and $H_2$ (resp., $L_1$ and $L_2$) and $X_0$ (resp. $Y_0$) is a vector field tangent to both foliations. As consequence, the PSVF $Z_0=(X_0,Y_0)$ is tangent to the foliation $M_1$ and the \textit{piecewise} foliation $M_2$. Moreover, all orbits of $Z_0$ are closed and topologically equivalent to $S^1$.

We stress that, in general (where we admit sliding and escaping motion on $\Sigma$), is \textbf{false} the natural aim: {\it The piecewise smooth mapping $H=
	\frac{h + l}{2} +{\rm sign}(z) \frac{h - l}{2}$ is a first integral of the vector field $Z=(X,Y)$
	provided that $h$ and $l$ are smooth first integrals of $X$ and $Y$ respectively}. See \cite{CarvalhoTeixeira-JDE2016} for examples.

Also observe that $Z_0$ is such that $Z_0(x,y,z)=-Z_0(-x,-y,-z)$ and so, it is $\varphi$\textit{-reversible}, where $\varphi(x,y,z)=(-x,-y,-z)$.

Another important definition is the concept of equivalence between two PSVFs.
\begin{definition}\label{definicao sigma-equivalencia}
	Two PSVFs $Z=(X,Y), \,
	\widetilde{Z}=(\widetilde{X},\widetilde{Y}) \in \Omega$, where $\Omega$ be the set of all PSVF endowed with the $C^r$ product topology, defined in
	open sets $U, \, \widetilde{U}$ and with switching manifold $\Sigma$
	are \textbf{$\mathbf{\Sigma}$-equivalent} if there exists an
	orientation preserving homeomorphism $h: U \rightarrow
	\widetilde{U}$ that sends $U\cap\Sigma$ to
	$\widetilde{U}\cap\Sigma$, the orbits of  $X$ restricted to
	$U\cap\Sigma^+$ to the orbits of $\widetilde{X}$ restricted to
	$\widetilde{U}\cap\Sigma^+$, and the orbits of  $Y$ restricted to
	$U\cap\Sigma^-$ to the orbits of $\widetilde{Y}$ restricted to
	$\widetilde{U}\cap\Sigma^-$.
\end{definition}

Now we state the main results of the paper.
\begin{proposition}\label{teorema equivalencia}
Let $Z=(X,Y)$ be a PSVF defined in a compact $\mathcal{M}$ presenting a continuous of topological cylinders filled by periodic orbits, then $Z$ is $\Sigma$-equivalent to $Z_0$ given by \eqref{eq forma normal centro}.
\end{proposition}

\vspace{.3cm}

\noindent {\bf Theorem A.} {\it Let $Z_0$ be given by \eqref{eq forma normal centro}. For any neighborhood $\mathcal{W} \subset \Omega$ of
	$Z_0$ and for any integer $\mathcal{L}>0$, there exists
	$\widetilde{Z} \in \mathcal{W}$ such that $\widetilde{Z}$ has $\mathcal{L}$ $Z_0$-invariant planes. Moreover, in each plane there is a center of $Z_0$.} 

\vspace{.3cm}

\noindent {\bf Theorem B.}
{\it Let $Z_0$ be given by \eqref{eq forma normal centro}. For any neighborhood $\mathcal{W} \subset \Omega$ of
$Z_0$ and for any integer $k>0$, there exists
$\widetilde{Z} \in \mathcal{W}$ such that $\widetilde{Z}$ has $k$ isolated invariant topological cylinders filled by periodic orbits. The same holds if
$k=\infty$.}

\vspace{.3cm}

\noindent {\bf Theorem C.} {\it Let $Z_0$ be given by \eqref{eq forma normal centro}. For any neighborhood $\mathcal{W} \subset \Omega$ of
	$Z_0$ and for any integers $\mathcal{L}>0$ and $k>0$, there exists
	$\widetilde{Z} \in \mathcal{W}$ such that $\widetilde{Z}$ has $\mathcal{L}.k$ hyperbolic limit cycles. The same holds if
	$k=\infty$. Moreover, the stability of each limit cycle is obtained. See Figure \ref{Fig Traj}.}

\begin{figure}[h!]
	\begin{center}
		\begin{overpic}[width=5in]{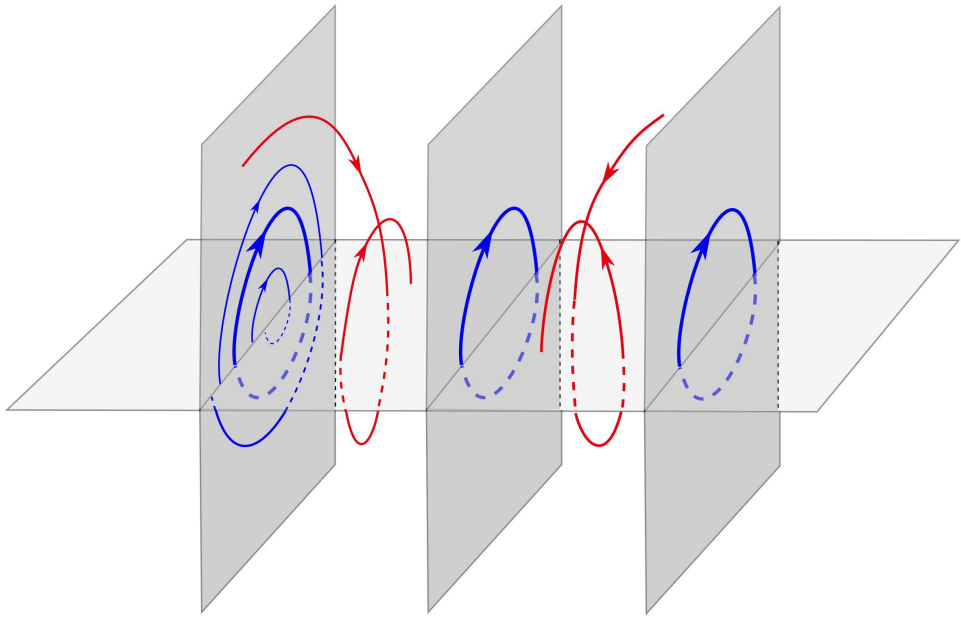}
		\end{overpic}
	\end{center}
	\caption{The trajectories according to Theorem C.}\label{Fig Traj}
\end{figure}

\vspace{.3cm}

Moreover, in the previous theorems,  we explicitly build families of PSVFs presenting the quoted properties. 

The paper is organized as follows. In Section \ref{secao preliminares} we introduce the terminology, some definitions and the basic theory about PSVFs. Sections \ref{secao teorema equivalencia}, \ref{secao prova teo A}, \ref{secao prova teo B} and \ref{secao prova teo C} are devoted to prove Proposition \ref{teorema equivalencia}, Theorem A, Theorem B and Theorem C, respectively.

\section{Preliminaries}\label{secao preliminares}


	\begin{definition}\label{definicao centro}
		Consider $Z \in \Omega$. We say that $q\in \Sigma$ is a \textbf{$\mathbf{\Sigma}$-center} of $Z$  if $q\in \Sigma$ and there is a codimension one manifold $\mathcal{S}$ such that $\Sigma \cap \hspace{-.28cm}| \,\, \mathcal{S}$ and there is a neighborhood $U\subset\R^3$ of $q$ where $U \cap \mathcal{S}$ is filled by a one-parameter family $\gamma_{s}$ of closed orbits of $Z$ in such a way that the orientation is preserved.
	\end{definition}

Consider the notation $X.f(p)=\left\langle \nabla f(p), X(p)\right\rangle$ and, for $i\geq 2$,  $X^i.f(p)=\left\langle \nabla X^{i-1}. f(p), X(p)\right\rangle$, where $\langle . , . \rangle$ is the usual inner product in $\R^3$. We say that a point $p \in \Sigma$ is a \textit{$\Sigma$-fold point} of $X$ if $X.f(p)=0$ but $X^{2}.f(p)\neq0.$ Moreover, $p\in\Sigma$ is a \textit{visible} (respectively {\it invisible}) $\Sigma$-fold point of $X$ if $X.f(p)=0$ and $X^{2}.f(p)> 0$ (respectively $X^{2}.f(p)< 0$). We say that $p \in \Sigma$ is a \textit{two-fold singularity} of $Z$ if $p$ is a $\Sigma$-fold point for both $X$ and $Y$. In this work, we consider only two-fold singularities of type invisible-invisible, ie, the fold points are invisible for both, $X$ and $Y$.

\begin{remark}\label{obs posicao dobras}Since $f(x,y,z)=z$, we conclude from \eqref{eq forma normal centro} that $L=\{(x,0,0) \, | \, x  \in \R \} \subset \Sigma$ is the curve of invisible fold singularities of both $X_0$ and $Y_0$.\end{remark}

Consider the case when the PSVF $Z=(X,Y)$ has $q$ as two-fold singularity.  
We can define the \textit{positive
half-return map} as $\varphi_X(\rho)=\rho^+$, and the \textit{negative
 half-return map} as $\varphi_Y(\rho^+)=\rho^-$ (see Figure \ref{1retorno}). The
complete \textit{return map} associated to $Z$ is given by the
composition of these two maps
\begin{equation}\label{eq:4}
 \varphi_Z(\rho)=\varphi_Y(\varphi_X(\rho)).
\end{equation}

\begin{figure}[h!]
\begin{center}
\begin{overpic}[width=2.5in]{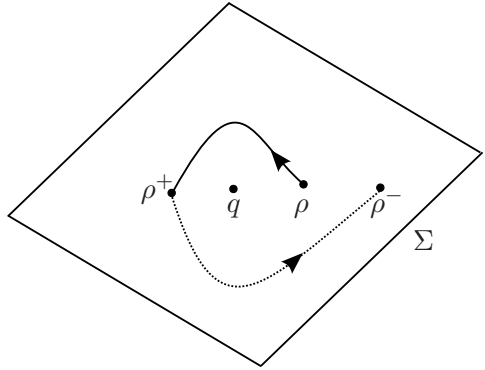}
\put(46,33){$q$}\put(60,33){$\rho$}\put(28,36){$\rho^+$}\put(76,33){$\rho^-$}\put(85,25){$\Sigma$}
\end{overpic}
\end{center}
\caption{Return map of $Z=(X,Y)$.}\label{1retorno}
\end{figure}

%
\begin{proposition}\label{propiedades}
The PSVF $Z_0=(X_0,Y_0)$ given by~\eqref{eq forma normal centro} has a continuous of topological cylinders and, in each cylinder, all orbits are periodic (see Figure~\ref{cilindro invariante}). Moreover, in each plane $\pi_M=\{(x,y,z) \, | \, x = M\}$, there is a $\Sigma$-center.
\end{proposition}
\begin{proof}
For a direct integration, the trajectories of $X_0$ and $Y_0$ are parametrized by 
\begin{equation}\label{eq param X0}\phi_{X_{0}}(t)=(x_0, -t+y_0, -t^2+2ty_0)\end{equation}and \begin{equation}\label{eq param Y0}\phi_{Y_{0}}(t)=(x_1, t+y_1, t^2+2ty_1),\end{equation}respectively.
Note that $\phi_{X_{0}}(0)=(x_0,y_0,0)$ and $\phi_{Y_{0}}(0)=(x_1,y_1,0)$. Thus the positive half-return map is
$\varphi_{X_0}(x,y)=(x,-y).$
Analogously, the negative half-return map is
$\varphi_{Y_0}(x,y)=(x,-y).$ Therefore, the
complete \textit{return map} associated to $Z_0$ is given by
\[ \varphi_{Z_0}(x,y)=\varphi_{Y_0}(\varphi_{X_0}(x,y))=(x,y).\]
\end{proof}

Note that by Proposition~\ref{propiedades}, we get $\varphi_{Z_0}(x,y)=(\varphi_{Z_0}^1(x),\varphi_{Z_0}^2(y))$, where $\varphi_{Z_0}^1(x)=x$ and $\varphi_{Z_0}^2(y)=y$. In order to obtain isolated $Z_0$-invariant planes we perturb the map $\varphi_{Z_0}^1(x)$ (see Theorem A), and in order to obtain isolated $Z_0$-topological cylinders we perturb the map $\varphi_{Z_0}^2(y)$ (see Theorem B). When we perturb both we are able to obtain hyperbolic limit cycles (see Theorem C).

\vspace{.5cm}

\begin{remark}$\;$
\begin{itemize}
\item In this work we decide consider only perturbations of \eqref{eq forma normal centro} that keep the straight line $L=\{(x,y,z)\in\mathbb{R}^3;y=z=0\}$ as a two-fold singularity. This assumption is important because in this case the return map is always well defined.

\item In this sense the return map of all trajectories considered in this paper
is given by the composition of two involutions (see \cite{Marco-enciclopedia}). 

\end{itemize}
\end{remark}

\section{Proof of Proposition \ref{teorema equivalencia}}\label{secao teorema equivalencia}

In this section we construct homeomorphism that sends orbits of $Z=(X,Y)$, that has a continuous of topological cylinders filled by periodic orbits, to orbits of
$Z_0=(X_0,Y_0)$ given by \eqref{eq forma normal centro}.

Without loss of generality consider that orbits of $Z$ are oriented in an anti-clockwise sense.
Let $L$ (respectively, $\overline{L}$) be a set of two-fold singularity of $Z_0$ (respectively, $Z$) with length $\mathcal{R}_1>0$. By arc length parametrization we identify $L$ with $\overline{L}$. By $p$ (respectively, $\overline{p}$), we mark the line segment $\mathcal{S}_p$ (respectively, $\mathcal{S}_{\overline{p}}$) of length $\mathcal{R}_2$ orthogonal to $\Sigma$ (see Figure~\ref{sigmaequi}). This segment reaches a topological cylinder $\mathcal{M}$ of $Z_0$ (respectively, $\overline{\mathcal{M}}$ of $Z$) at a point $p^1$ (respectively, $\overline{p}^1$). 

In each point $\alpha \in L=[p,r]$ (respectively, $\overline{\alpha} \in \overline{L}$) mark the line segment $\mathcal{S}_\alpha$ orthogonal to $\Sigma$ (respectively, $\mathcal{S}_{\overline{\alpha}}$) with final point in $\mathcal{M}$ (respectively, $\overline{\mathcal{M}}$). Once $L$ and $\overline{L}$ are identified, identify each $S_\alpha$ with $\mathcal{S}_{\overline{\alpha}}$ by arc length parametrization. 

By the Implicit Function Theorem (abbreviated by IFT), there exists a smallest time $t_{1}< 0$ (respectively, $\overline{t}_{1} < 0$), depending on $p^1$ (respectively, $\overline{p}^1$), such that $\phi_{X_0}(p^1,t_{1}):= q \in \Sigma(+)$ (respectively, $\phi_{X}(\overline{p}^1,\overline{t}_{1}):=
\overline{q} \in \overline{\Sigma}(+)$), where $\Sigma(+)$ (respectively, $\overline{\Sigma}(+)$) is the set of all points of $\Sigma$ situated on the right of $L$ (respectively, $\overline{L}$) and $\phi_{W}$ denotes the flow of the vector field $W$. Identify the orbit arcs $\gamma_{q}^{p^1}(X_0)$ and $\gamma_{\overline{q}}^{\overline{p}^1}(X)$ of $X_0$ and $X$
with initial points $q$ and $\overline{q}$ and final points $p^1$ and $\overline{p}^1$, respectively, by arc length parametrization. Again  by IFT,  there exists a smallest time $t_{2}> 0$ (respectively, $\overline{t}_{2} > 0$), depending on $p^1$ (respectively, $\overline{p}^1$), such that $\phi_{X_0}(p^1,t_{2}):= q^{1} \in \Sigma(-)$ (respectively, $\phi_{X}(\overline{p}^1,\overline{t}_{2}):=
\overline{q}^{1} \in \overline{\Sigma}(-)$)
where $\Sigma(-)$ (respectively, $\overline{\Sigma}(-)$) is the set of all points of $\Sigma$ situated on the left of $L$ (respectively, $\overline{L}$). Identify the orbit arcs $\gamma_{p^1}^{q^1}(X_0)$ and $\gamma_{\overline{p}^1}^{\overline{q}^1}(X)$ of $X_0$ and $X$
with initial points $p^1$ and $\overline{p}^1$ and final points $q^1$ and $\overline{q}^1$, respectively, by arc length parametrization.


\begin{figure}[h!]
	\begin{center}
	\begin{overpic}[width=3.5in]{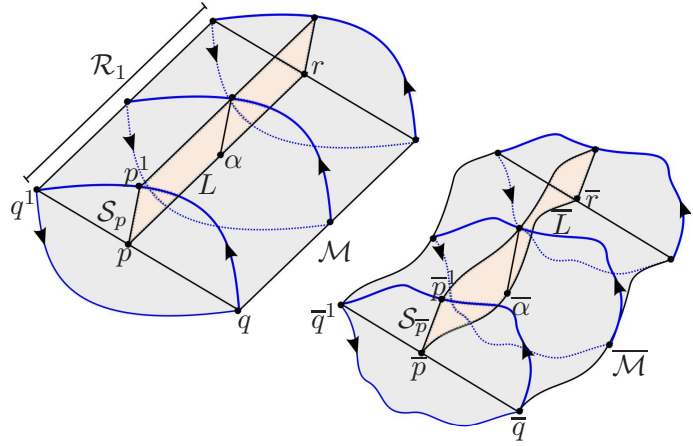}
\put(10,53){$\mathcal{R}_1$}\put(14,25){$p$}\put(58,8){$\overline{p}$}\put(11,31){$\mathcal{S}_p$}\put(56,15){$\mathcal{S}_{\overline{p}}$}\put(32,15){$q$}\put(73,-1){$\overline{q}$}\put(44,25){$\mathcal{M}$}\put(88,8){$\overline{\mathcal{M}}$}\put(26,35){$L$}\put(79,29){$\overline{L}$}\put(-2,32){$q^{1}$}\put(43,15){$\overline{q}^{1}$} \put(15,37){$p^1$}\put(61,20){$\overline{p}^1$}\put(43,53){$r$} 
\put(84,33){$\overline{r}$} 
\put(30,39){$\alpha$}\put(73,17){$\overline{\alpha}$}
		\end{overpic}
	\end{center}
	\caption{Topological cylinders.}\label{sigmaequi}
\end{figure}

Now, since $Z_0$ (respectively, $Z$) presents a continuous of topological cylinders, and
$L$ (respectively, $\overline{L}$) is an invisible $\Sigma$-fold set $Y_0$ (respectively, $Y$) by the IFT, there exists a smallest time $t_{3}
> 0$ (respectively, $\overline{t}_{3} > 0$), depending on $q^1$ (respectively, $\overline{q}^1$), such that $\phi_{Y_0}(q^1,t_{3}):= q\in \Sigma(+)$ (respectively, $\phi_{Y}(\overline{q}^1,\overline{t}_{3}):= \overline{q}\in \overline{\Sigma}(+)$). Identify the orbit arcs $\gamma_{q^1}^{q}(Y_0)$ and $\gamma_{\overline{q}^1}^{\overline{q}}(Y)$ of $Y_0$ and $Y$
with initial points $q^1$ and $\overline{q}^1$ and final points $q$ and $\overline{q}$, respectively, by arc length parametrization. 

Do the same for all point $\beta \in S_{\alpha}$ (resp., $\overline{\beta} \in S_{\overline{\alpha}}$), and for all $\alpha\in L$ (resp., $\overline{\alpha}\in \overline{L}$).

\section{A perturbation on the horizontal axis $-$ Proof of Theorem A}\label{secao prova teo A}

%
%
%
%
%

Now we consider a perturbation on the normal form \eqref{eq forma normal centro}  that keeps invariant the nested cylinders and exactly $\mathcal{L}$ planes of the form $\pi_i=\{(x,y,z) \, | \, x =  i \mu \}$, where $i\in \{0,1,2,\hdots,\mathcal{L}-1\}$ and $\mu>0$ is a small real number. In fact, consider
\begin{equation}\label{eq pert X}
\overline{X}_{\mathcal{L}}(x,y,z) = \left(
\begin{array}{c}
x (x - \mu) (x-2 \mu) \hdots (x-(\mathcal{L}-1) \mu) \\
0  \\
0
\end{array}
\right) = 
\end{equation}$$=\left(
\begin{array}{c}
\Pi _{i=0}^{\mathcal{L}-1} (x-i \mu) \\
0  \\
0
\end{array}
\right)$$and
\begin{equation}\label{eq forma normal Z com pert X}
Z_\mathcal{L}(x,y,z)= \left\{
\begin{array}{ccl}
X_\mathcal{L}(x,y,z)= & \left(
\begin{array}{c}
\lambda \Pi_{i=0}^{\mathcal{L}-1} (x-i \mu) \\
-1  \\
2 y
\end{array}
\right) 
& \hbox{if $z \geq 0$,} \\
Y_0(x,y,z) = & \left(
\begin{array}{c}
0 \\
1  \\
2 y
\end{array}
\right)& \hbox{if $z \leq 0$,}
\end{array}
\right.
\end{equation}where $X_\mathcal{L}(x,y,z)=X_0(x,y,z) + \lambda \overline{X}_{\mathcal{L}}(x,y,z)$, with $\lambda$ a sufficiently smal real number and $X_0$ given in \eqref{eq forma normal centro}.

\begin{remark}
 There is nothing special in the set $\{0,1,2,\hdots,\mathcal{L}-1\}$ of  sequential positive integers and we could take any set of $\mathcal{L}$ integers in the previous consideration. 
\end{remark}


\begin{proposition}\label{prop cilindros invariantes com pert X}
	The topological cylinders obtained in Proposition \ref{propiedades} are  $Z_\mathcal{L}$-invariant.
\end{proposition}
\begin{proof}
By Remark \ref{obs posicao dobras}, $L=\{(x,0,0) \, | \, x  \in \R \} \subset \Sigma$ is the curve of invisible fold singularities of both $X_0$ and $Y_0$.

The positive half-return map is
$\varphi_{X_\mathcal{L}}(x,y)=(\varphi_{X_\mathcal{L}}^1(x),-y)$ and the negative half-return map is
$\varphi_{Y_0}(x,y)=(x,-y).$ 	
For a fixed $\beta \in \R$, take $L_\beta=\{(x,\beta,0) \, | \, x  \in \R \} \subset \Sigma$ and let us saturate this straight line by the $Z_\mathcal{L}$-flow.
	In fact, for all $(x,\beta,0) \in L_{\beta}$ we get
\begin{equation}\label{eq invariancia}\begin{array}{cl} \varphi_{Z_{\mathcal{L}}}(x,\beta,0)&=	\varphi_{Y_{0}}(\varphi_{X_{\mathcal{L}}}(x,\beta,0))=(\varphi_{X_\mathcal{L}}^1(x),\beta,0)\in L_\beta.\end{array}\end{equation}
\end{proof}

\begin{proposition}\label{prop planos invariantes com pert X}
	The planes $\pi_i=\{(x,y,z) \, | \, x =  i \mu \}$, where $i\in \{0,1,2,\hdots,\mathcal{L}-1\}$, are $Z_\mathcal{L}$-invariant.
\end{proposition}
\begin{proof}
Take $i=i_0$ fixed. When $x=i_0 \mu$ we get that the first coordinate of  $X_\mathcal{L}$ is null. As consequence,   the plane $\pi_{i_{0}}$ is $X_\mathcal{L}$-invariant. The same holds for all $i\in \{0,1,2,\hdots,\mathcal{L}-1\}$. On the other hand,  for all $c\in \R$, the plane $\pi_c=\{(x,y,z) \, | \, x =  c \}$ is $Y_0$-invariant. Therefore, each plane $\pi_i$ is $Z_\mathcal{L}$-invariant.
\end{proof}

\begin{proposition}\label{prop ZL com centro}
	The PSVF $Z_\mathcal{L}$ has a $\Sigma$-center in each plane $\pi_i$, where $i\in \{0,1,2,\hdots,\mathcal{L}-1\}$.  
\end{proposition} 
\begin{proof}
	The proof is straighforward. Is enough to combine Propositions \ref{prop cilindros invariantes com pert X} and \ref{prop planos invariantes com pert X}.
\end{proof}

\begin{proposition}\label{prop estabilidade centros} 

		
		
		
	When $i$ is even (resp. odd) the $\Sigma$-center $\pi_{i}$ behaves like a unstable (resp. stable) center manifold where $i\in \{0,1,2,\hdots,\mathcal{L}-1\}$.
\end{proposition}
\begin{proof}
From Proposition \ref{prop planos invariantes com pert X} the planes $\pi_i=\{(x,y,z) \, | \, x =  i \mu \}$ are $Z_\mathcal{L}$-invariant. As stated in Equation \ref{eq invariancia}, we get
$\varphi_{Z_\mathcal{L}}(x,y,0)=(\varphi_{X_\mathcal{L}}^1(x),y,0)$. As consequence, the behavior of the complete return map is determined by $\varphi_{X_\mathcal{L}}^1(x)$. So, let us consider the first coordinate of $X_\mathcal{L}$, i.e., let us consider the differential equation
\begin{equation}\label{eqx}
\dot{x}=\Pi_{i=0}^{\mathcal{L}-1} (x-i \mu).
\end{equation}
Note that each $x=i\mu$, $i\in \{0,1,2,\hdots,\mathcal{L}-1\}$, is a solution of \eqref{eqx} and 
\[
\dfrac{d}{dx}\Pi_{i=0}^{\mathcal{L}-1} (x-i \mu)|_{x=i\mu}
\]
is positive for $i$ even and negative for $i$ odd.
The behavior in each solution is given in the Figure \ref{grafico}.

\begin{figure}[h!]
	\begin{center}
	\begin{overpic}[width=5in]{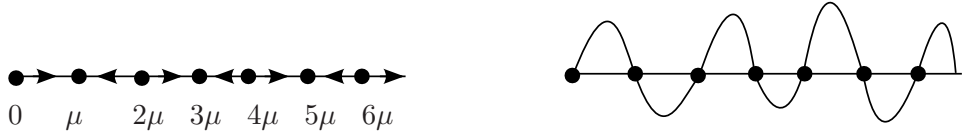}
\put(0,0){$0$}\put(6,0){$\mu$}\put(13,0){$2\mu$}\put(19,0){$3\mu$}\put(25,0){$4\mu$}\put(31,0){$5\mu$}\put(37,0){$6\mu$}
		\end{overpic}
	\end{center}
	\caption{ The phase portrait of \eqref{eqx} and the graph of $y=\Pi_{i=0}^{\mathcal{L}-1} (x-i \mu)$.}\label{grafico}
\end{figure}


Thus,	when $i$ is even (resp. odd) the $\Sigma$-center $\pi_{i}$ behaves like a unstable (resp. stable) center manifold, where $i\in \{0,1,2,\hdots,\mathcal{L}-1\}$.

\end{proof}

\begin{proof}[Proof of Theorem A]
	The Propositions \ref{prop planos invariantes com pert X}, \ref{prop ZL com centro} and \ref{prop estabilidade centros} prove Theorem A. 
\end{proof}

\section{A perturbation of the continuum of cylinders $-$ Proof of Theorem B}\label{secao prova teo B}

In order to prove Theorem B we need some lemmas. Observe that both vector fields $X_0$ and $Y_0$ in the normal form \eqref{eq forma normal centro} are written as $W(x,y,z)=(0,\pm 1, g(y))$ (particularly, $g(y)= 2y$ in such expression). Next lemma gives how are the trajectories of such systems.
\begin{lemma}\label{lema trajetorias}
The trajectories of a vector field $W(x,y,z)=(0,1,g(y))$, in each plane $\pi_c=\{(x,y,z) \, | \, x =c \, ; \, c \in \R  \}$, are obtained by vertical translations of the graph of $G(y)$, where $\frac{\partial }{\partial y}G(y)=g(y)$.
\end{lemma}
\begin{proof}
Since $W(x,y,z)=(\dot{x}, \dot{y},\dot{z})=(0,1, g(y)) \in \chi^r$ we obtain that \[x(t)=c_1, \mbox{ } y(t)=t+c_2 \mbox{ and } z(t) = \int g(t+c_2) dt= G(t+c_2)+c_3,\]where $c_1,c_2,c_3 \in \R$ and $G$ is a primitive of $g$. Now, take $u=t+c_2$ and the trajectories of $W(x,y,z)$ are given by $(c_1,u,G(u) + c_3)$ which in each plane $\pi_{c_{1}}=\{(x,y,z) \, | \, x =c_1  \}$, are vertical translations of the graph of $G(u)$.
\end{proof}
Observe that an analogous result is obtained with $W(x,y,z)=(0,-1,g(y))$.

In what follows, $h : \R \rightarrow \R$ will denote the $C^\infty$-function given by
\[h(y)=\left\{\begin{array}{ll}
0, & \mbox{ if }y\leq0,\\
e^{-1/y}, &\mbox{ if }y>0.
\end{array}\right.\]
\begin{lemma}\label{lema finitos pontos fixos trajetorias}
Consider the  function $\xi_\varepsilon^f(y)= \varepsilon h(y)
(\varepsilon-y)(2\varepsilon-y)\dots(k\varepsilon-y)$.
\begin{enumerate}
\item[(i)] If $\varepsilon<0$ then $\xi_\varepsilon^f$ does not have roots in $(0,+\infty)$.
\item[(ii)] If $\varepsilon>0$ then $\xi_\varepsilon^f$
has exactly $k$ roots in $(0,+\infty)$, these roots are
$\{\varepsilon,2\varepsilon\dots,k\varepsilon\}$ and
$\displaystyle\frac{\partial \xi_\varepsilon^f}{\partial
y}(j\varepsilon)=(-1)^j\varepsilon^kh(j\varepsilon)(k-j)! (j-1)!
\mbox{ for }j\in\{1,2,\dots,k\}.$ It means that the derivative at
the root $j\varepsilon$ is positive for $j$ even and negative for $j$
odd.
\end{enumerate}
\end{lemma}
\begin{proof}
When $y>0$, by a straightforward calculation $\xi_\varepsilon^f(y)=0$
if, and only if, $(\varepsilon-y)(2\varepsilon-y)\dots(k\varepsilon-y)=0$.
So, the roots of $\xi_\varepsilon^f(y)$ in $(0,+\infty)$  are
$\varepsilon,2\varepsilon,\dots,k\varepsilon$. Moreover,
$$\displaystyle\frac{\partial \xi_\varepsilon^f}{\partial y}(y)=
 \displaystyle\frac{\partial }{\partial y}\Big((j \varepsilon - y) H(y) \Big) =
 (j \varepsilon - y) \displaystyle\frac{\partial H}{\partial y}(y) -  H(y),$$ where $H(y) = \xi_\varepsilon^f(y)/(j \varepsilon -
 y)$. So,
$$\begin{array}{ccc}
  \displaystyle\frac{\partial \xi_\varepsilon^f}{\partial y}(j \varepsilon) & = & -  H(j \varepsilon) = \varepsilon^{k} h(j \varepsilon) (1-j)\dots((j-1)-j )((j+1)-j)\dots(k-j) \\
    & = & \varepsilon^{k} h(j \varepsilon) (-1)^{j} \Big((j-1)\dots(j- (j-1))\Big) \Big(((j+1)-j)\dots(k-j)\Big) \\
    & = & (-1)^j\varepsilon^kh(j\varepsilon)(k-j)! (j-1)! \\
 \end{array}$$
This proves item (ii). Item (i) follows immediately.
\end{proof}
\begin{lemma}\label{lema infinitos pontos fixos trajetorias}
Consider the  function $\xi_\varepsilon^i(y)= - h(y) \sin(\pi \varepsilon^2 /y)$.  For $\varepsilon\neq0$ the function $\xi_\varepsilon^i$ has infinity many
roots in $(0,\varepsilon^2)$, these roots are $\{\varepsilon^2,\varepsilon^2/2,\varepsilon^2/3,\dots\}$ and
$$\displaystyle\frac{\partial \xi_\varepsilon^i}{\partial
y}(\varepsilon^2/j)=(-1)^{j+1}  (-\pi j^2/\varepsilon^2)h(\varepsilon^2/j) \mbox{ for
}j\in\{1,2,3,\dots\}.$$ It means that the derivative at the root $\varepsilon^2/j$
is positive for $j$ even and negative for $j$ odd.
\end{lemma}
\begin{proof}
When $y>0$, by a straightforward calculation $\xi_\varepsilon^i(y)=0$
if, and only if, $\sin(\pi \varepsilon^2/y)=0$. So, the roots of
$\xi_\varepsilon^i(y)$ in $(0,\varepsilon^2)$  are $\varepsilon^2,\varepsilon^2/2,\varepsilon^2/3,\dots$. Moreover,
$$\displaystyle\frac{\partial \xi_\varepsilon^i}{\partial y}(y)=
- h^{\prime}(y) \sin(\pi \varepsilon^2/y)  - h(y)\cos(\pi \varepsilon^2/y)(-\pi\varepsilon^2/y^2).$$ So,
$$\begin{array}{rcl}
  \displaystyle\frac{\partial \xi_\varepsilon^i}{\partial y}(\varepsilon^2/j) & = & -h^{\prime}(\varepsilon^2/j) \sin(\pi j)  - h(\varepsilon^2/j)\cos(\pi j)(-\pi j^2/\varepsilon^2) \\
&    = & (-1)^{j+1} (-\pi j^2/\varepsilon^2) h(\varepsilon^2/j). \\
 \end{array}$$
\end{proof}
Since $h$ is a C$^\infty$-function, the functions $\xi_\varepsilon^f(y)$ in
Lemma \ref{lema finitos pontos fixos trajetorias}  and $\xi_\varepsilon^i(y)$ in Lemma \ref{lema infinitos pontos fixos trajetorias} are C$^{\infty}$-functions. So
$Z_{\varepsilon}^{\rho} \in \Omega$ given by
\begin{equation}\label{eq centro perturbado}
Z^{\rho}_{\varepsilon}(x,y,z) = \left\{
      \begin{array}{ll}
        X_0(x,y,z) = \left(
              \begin{array}{c}
              0\\
                   -1 \\
               2y
               \end{array}
      \right)
 & \hbox{if $z \geq 0$,} \\
         Y^{\rho}_{\varepsilon}(x,y,z) = \left(
              \begin{array}{c}
              0\\
               1 \\
               2y + \frac{\partial \xi_\varepsilon^\rho}{\partial y} (y)
\end{array}
      \right)& \hbox{if $z \leq 0$,}
      \end{array}
    \right.
\end{equation}where either $\rho=f$ or $\rho =i$, is a small C$^{\infty}$-perturbation of $Z_0$ given by \eqref{eq forma normal centro} when $\varepsilon$ is sufficiently small. Moreover,
\begin{equation}\label{convergencia}
\displaystyle\lim_{\varepsilon \rightarrow 0} Z^{\rho}_{\varepsilon} = Z_0.
\end{equation}
\begin{lemma}\label{lema primeiro retorno}Let
$\varphi_{Z^{\rho}_{\varepsilon}}(x,y)=(\varphi_{Z^{\rho}_{\varepsilon}}^1(x),\varphi_{Z^{\rho}_{\varepsilon}}^2(y))$ be the return map of $Z^{\rho}_{\varepsilon}$ where either
$\rho=f$ or $\rho=i$. For all $y>0$ we have that
\[ y^2-(\varphi_{Z^{\rho}_{\varepsilon}}^2(y))^2-\xi_\varepsilon^\rho(\varphi_{Z^{\rho}_{\varepsilon}}^2(y))=0.\]
\end{lemma}

\begin{proof}
Let $(x_0,y_0,0)\in\Sigma$. According to Lemma \ref{lema trajetorias}, in each plane $\pi_{x_0}=\{(x,y,z) \, | \, x =x_0\}$, the trajectories of $X_0$ are the graphs of $F_c(y)=-y^2+c$ for $c\in\R$. The constant $c\in\R$ that satisfy $F_c(y_0)=0$ is $c=y_0^2$. The parabola
$z=-y^2+y_0^2$ in the plane $\pi_{x_0}$ intersects the plane $z=0$ at the points $(x_0,y_0,0)$ and $(x_0,-y_0,0)$. So, $\varphi_{X_0}(x_0,y_0)=(x_0,-y_0)$ and thus $\varphi_{X_0}^2(y_0)=-y_0$. Again by Lemma \ref{lema trajetorias}, in each plane $\pi_{x_0}$, the trajectories of
$Y^\rho_\varepsilon$ are the graphs of
$G_c(y)=y^2+\xi_\varepsilon^\rho(y)+c$ for $c \in \R$. The constant $c \in\R$ that satisfy $G_c(-y_0)=0$ is $c = -y_0^2$. So, in the plane $\pi_{x_0}$, the first
return $\varphi_{Y^\rho_\varepsilon}^2(-y_0)$ is the first coordinate of the
point in $\Sigma$ given by the intersection of the graph of the
function $z=G(y)=y^2+\xi_\varepsilon^\rho(y)-y_0^2$ with the plane $z=0$. So
$\varphi_{Z^{\rho}_{\varepsilon}}^2(y)$ satisfies
\begin{equation}\label{eq1}
y^2-(\varphi_{Z^{\rho}_{\varepsilon}}^2(y))^2-\xi_\varepsilon^\rho(\varphi_{Z^{\rho}_{\varepsilon}}^2(y))=0,
\end{equation}where either
$\rho=f$ or $\rho=i$.
\end{proof}

\begin{lemma}\label{lema derivada aplicacao retorno}
Let $\varphi_{Z^{f}_{\varepsilon}}^2$ be the second component of return map of $Z^{f}_{\varepsilon}$. Then $y>0$ is a fixed point of $\varphi_{Z^{f}_{\varepsilon}}^2$ if, and only if,
$y=j\varepsilon$ for $j=1,2,\dots,k$. Moreover, for $j$ even
$(\varphi_{Z^{f}_{\varepsilon}}^2)^\prime(j\varepsilon)<1$ and for $j$ odd
$(\varphi_{Z^{f}_{\varepsilon}}^2)^\prime(j\varepsilon)>1$. \end{lemma}

\begin{proof}
According to Lemma \ref{lema primeiro retorno},
$y=\varphi_{Z^{f}_{\varepsilon}}^2(y)$ if, and only if,
$\varphi_{Z^{f}_{\varepsilon}}^2(y)$ is a zero of the function $\xi_\varepsilon^f(y)$, i.e., by Lemma \ref{lema finitos pontos fixos
trajetorias}, $y=j\varepsilon$ for $j=1,2,\dots,k$. Differentiating
\eqref{eq1} with respect to $y$ we obtain
$2y-2\varphi_{Z^{f}_{\varepsilon}}^2(y)(\varphi_{Z^{f}_{\varepsilon}}^2)^\prime(y) -
\frac{\partial \xi_\varepsilon^f}{\partial
y}(\varphi_{Z^{f}_{\varepsilon}}^2(y))(\varphi_{Z^{f}_{\varepsilon}}^2)^\prime(y)=0$,
and so
\[(\varphi_{Z^{f}_{\varepsilon}}^2)^\prime(j\varepsilon) = \frac{2j\varepsilon}{2j\varepsilon+\frac{\partial \xi_\varepsilon^f}{\partial y}(j\varepsilon)}.\]
According to Lemma \ref{lema finitos pontos fixos trajetorias}, if $j$ is even then $\frac{\partial \xi_\varepsilon^f}{\partial
y}(j\varepsilon)>0$ and it implies that
$(\varphi_{Z^{f}_{\varepsilon}})^\prime(j\varepsilon)<1$. And if $j$ is odd
then $(\varphi_{Z^{f}_{\varepsilon}})^\prime(j\varepsilon)>1$.
\end{proof}
\begin{remark}\label{observacao repetir lema 8 para infinitos} A similar result obtained in Lemma \ref{lema derivada aplicacao retorno}, for the PSVF $Z^i_\varepsilon$, also holds. \end{remark}
With the previous lemmas we can stated the following proposition.

\begin{proposition}\label{proposicao estabilidade ciclos}
Consider $Z^{\rho}_\varepsilon$ given by \eqref{eq centro perturbado}.
Then, for $\varepsilon = 0$, $Z^{\rho}_{\varepsilon}=Z_0$ given by \eqref{eq forma normal centro} has a continuous of topological cylinders and
\begin{itemize}
\item[(I)] For $\rho=f$
\begin{itemize}
\item[(I.i)] $Z^f_{\varepsilon}$ has $k$ isolated topological cylinders when $\varepsilon > 0$,

\item[(I.ii)] The topological cylinder passing through $y=j \varepsilon,z=0$ is
attractor (respectively, repeller) if $j$ is even (respectively,
odd), with $j \in \{  1,2,\dots k\}$.
\end{itemize}

\item[(II)] For $\rho=i$
\begin{itemize}
\item[(II.i)]  $Z^i_{\varepsilon}$ has infinitely many isolated topological cylinders when $\varepsilon \neq 0$,

\item[(II.ii)] The invariant cylinder passing through $y=\varepsilon^2/j,z=0$ is attractor
(respectively, repeller) if $j$ is even (respectively, odd).
\end{itemize}
\end{itemize}
\end{proposition}

\begin{proof}
According to Lemma \ref{lema primeiro retorno},
$y=\varphi_{Z_{\varepsilon}^{\rho}}^2(y)$ if, and only if,
$\varphi_{Z_{\varepsilon}^{\rho}}^2(y)$ is a zero of the function $\xi_\varepsilon^\rho(y)$.

Therefore when $\rho=f$, by Lemma \ref{lema finitos pontos fixos trajetorias},  the fixed points of $\varphi_{Z_{\varepsilon}^{f}}^2$ are
given by $y=j\varepsilon$ for $j=1,2,\dots,k$. Since an isolated fixed
point of $\varphi_{Z_{\varepsilon}^{f}}^2$ corresponds to a hyperbolic invariant cylinder of $Z_{\varepsilon}^{f}$, items (I.i) and (I.ii) follow
immediately from Lemma \ref{lema finitos pontos fixos trajetorias} (item (ii)), and Lemma \ref{lema derivada aplicacao retorno}.

On other hand when $\rho=i$, by Lemma \ref{lema infinitos pontos
fixos trajetorias},  the fixed points of $\varphi_{Z_{\varepsilon}^{i}}^2$
are given by $y=\varepsilon^2/j$ for $j=1,2,3,\dots$. Since an isolated fixed
point of $\varphi_{Z_{\varepsilon}^{i}}^2$ corresponds to a hyperbolic invariant cylinder of $Z_{\varepsilon}^{i}$, items (II.i) and (II.ii) follow
immediately from Lemma \ref{lema infinitos pontos fixos trajetorias}
and Remark \ref{observacao repetir lema 8 para infinitos}.
\end{proof}
Finally, we can prove Theorem B.
\begin{proof}[Proof of Theorem B]
Let ${\mathcal W}\subset\Omega$ be an arbitrary neighborhood of $Z_0$. According to \eqref{convergencia}, for $\varepsilon>0$ sufficiently small we have that $Z^\rho_\varepsilon \in {\mathcal W}$. The conclusion of the proof follows from Proposition \ref{proposicao estabilidade ciclos} just taking $\widetilde{Z}=Z^\rho_\varepsilon$.
\end{proof}

\section{Combining the two previous perturbations $-$ Proof of Theorem C}\label{secao prova teo C}

Now we combine the perturbations \eqref{eq forma normal Z com pert X} and \eqref{eq centro perturbado} of the normal form \eqref{eq forma normal centro} given in the two previous sections in order to obtain Theorem C. In fact, it gives rise to the following PSVF

\begin{equation}\label{eq forma normal Z com pert X e em Y}
Z_{k \mathcal{L}}(x,y,z)= \left\{
\begin{array}{ccl}
X_\mathcal{L}(x,y,z)= & \left(
\begin{array}{c}
\Pi_{i=0}^{\mathcal{L}-1} (x-i \mu) \\
-1  \\
2 y
\end{array}
\right) 
& \hbox{if $z \geq 0$,} \\
Y_k (x,y,z) = & \left(
\begin{array}{c}
0 \\
1  \\
2 y + \frac{\partial \xi_\varepsilon^\rho}{\partial y} (y)
\end{array}
\right)& \hbox{if $z \leq 0$.}
\end{array}
\right.
\end{equation}where $i\in \{0,1,2,\hdots,\mathcal{L}-1\}$, either $\rho=f$ or $\rho =i$, $\xi_\varepsilon^\rho$ is given in the previous section and $\mu,\varepsilon\in\R$ are small numbers.

\begin{proof}[Proof of Theorem C]
	First of all note that the two perturbations considered are uncoupled. 
	
	Theorem A ensures the existence of exactly $\mathcal{L}$ $Z_{k \mathcal{L}}$-invariant planes $\pi_i$. Moreover, the Proposition \ref{prop estabilidade centros} guarantees that these planes are repellers (resp., attractors) for $i$ even (resp., odd).
	
	Theorem B ensures the existence of exactly $k$ $Z_{k \mathcal{L}}$-invariant topological cylinders. Moreover, items I.ii and II.ii of Proposition \ref{proposicao estabilidade ciclos} guarantees that these nested cylinders are repellers (resp., attractors) for $j$ odd (resp., even), where $j=1,2, \hdots, k$.
	
	The intersection of the $\mathcal{L}$ planes of Theorem A and the $k$ cylinders of Theorem B, gives rise to the born of $k.\mathcal{L}$ limit cycles. Moreover, Propositions \ref{prop estabilidade centros} and \ref{proposicao estabilidade ciclos} ensures that these limit cycles are hyperbolic. 
	
	The stability of the limit cycle living at the intersection of the plane $\pi_{i}$ with the cylinder $j$ is of attractor kind when $i$ is odd and $j$ is even, of repeller kind when $i$ is even and $j$ is odd and of saddle kind otherwise. 
\end{proof}

\noindent {\textbf{Acknowledgments.} T. Carvalho is partially supported by the CAPES grant number 1576689 (from the program PNPD) and also is grateful to the	 FAPESP/Brazil grant number 2013/34541-0, the CNPq-Brazil grant number 443302/2014-6 and the CAPES grant number 88881.030454/2013-01 (from the program CSF-PVE). 

\smallskip B. Freitas is partially supported by the
PROCAD-88881.068462/2014-01 and by the
PRONEX/CNPq/FAPEG-2012.10.26.7000.803.


\begin{thebibliography}{00}





\bibitem{Livro-Bluman} S.C. Anco and G.W. Bluman, 
{\it Symmetry and Integration Methods for Differential Equations}. Springer-Verlag New York, vol 154 (2002).



\bibitem{AP} D.K. Arrowsmith, C.M. Place, An introduction to dynamical systems, Cambridge University Press, 1990.



\bibitem{Livro-Bogolyubov} 	N. N. Bogolyubov and D. ter Haar, 
{\it 	A Method for Studying Model Hamiltonians. A Minimax Principle for Problems in Statistical Physics}. Pergamon Press (1972).





\bibitem{Brogliato} B. Brogliato,  ``Nonsmooth Mechanics: Models, Dynamics and Control'', Springer-Verlag, New York (1999).




\bibitem{Eu-fold-sela} C.A. Buzzi, T. Carvalho, M.A. Teixeira, On three-parameter families of Filippov systems $-$ The
Fold-Saddle singularity, International Journal of Bifurcation and
Chaos 22 (2012), no. 12, 1250291, 18pp.

\bibitem{Eu-fold-cusp} C.A. Buzzi, T. Carvalho, M.A. Teixeira, On $3$-parameter families of piecewise smooth vector fields in the plane, SIAM J. Applied Dymanical Systems 11 (2012), no. 4, 1402--1424.




\bibitem{Livro-CamachoNet} C. Camacho and A. L. Neto, 
{\it Geometric Theory of Foliations}. Birkhäuser (1984).


\bibitem{CarvalhoTeixeira-JDE2016} T. Carvalho and M.A. Teixeira. {\it On piecewise smooth vector fields tangent to nested tori},  Journal of Differential Equations, in press (2016).


\bibitem{Ca04} M. Caubergh, Limit cycles near centers, Thesis Limburgh University, Diepenbeck (2004).

\bibitem{CD04} M. Caubergh, F. Dumortier, Hopf-Takens bifurcations and centers, J. Differential Equations 202 (2004), 1--31.

\bibitem{Ne} F. Ceragioli,  Discontinuous Ordinary Differential Equations and Stabilization, PhD thesis, University of Firenze, Italy, 1999. Electronically available at
http://calvino.polito.it/~ceragioli.


\bibitem{CH} S.N. Chow, J.K. Hale, Methods of bifurcation theory, Springer-Verlag, 1982.



\bibitem{Livro-Chrusciel} 	P. T. Chrusciel, J. Jezierski and J. Kijowski, 
{\it 		A Hamiltonian field theory in the radiating regime}. Springer (2002).



\bibitem{Co} J. Cort\'es, Discontinuous dynamical systems: A tutorial on solutions, nonsmooth analysis, and stability, posted in arXiv:0901.3583 [math.DS].


\bibitem{diBern-livro} M. di Bernardo, C.J. Budd, A.R. Champneys,and P. Kowalczyk. {\it Piecewise-smooth dynamical systems $-$ Theory and applications},  vol. 163 of  Applied Mathematical Sciences.
Springer-Verlag London, Ltd., London (2008).


\bibitem{diBern-relay} M. di Bernardo, K.H. Johansson \& F.  Vasca,  ``Self-Oscillations and Sliding in Relay Feedback Systems: Symmetry and Bifurcations,'' Internat. J.  Bif. Chaos, vol. 11 (2001),  1121--1140.



\bibitem{Dixon} D.D. Dixon, ``Piecewise Deterministic Dynamics from the Application of Noise to Singular Equation of Motion,'' J. Phys A: Math. Gen. 28 (1995), 5539--5551.






\bibitem{Du} F. Dumortier, Singularities of vector fields, Monografias de Matematica, no. 32, Instituto de Matematica Pura e Aplicada, Rio de Janeiro, 1978.

\bibitem{Ek} I. Ekeland, Discontinuits de champs hamiltoniens et existence de solutions optimales en calcul des variations, Publ. Math. Inst. Hautes \'Etudes Sci. 47 (1977), 5--32 (in French).

\bibitem{Fi} A.F. Filippov, Differential Equations with Discontinuous Righthand Sides,
Mathematics and its Applications (Soviet Series), Kluwer Academic Publishers-Dordrecht, 1988.

\bibitem{Gav} L. Gavrilov, E. Horozov, Limit cycles of perturbations of quadratic Hamiltonian vector fields. Journal de Math\'ematiques Pures et Appliqu\'ees (9) 72 (1993), no. 2, 213--238.

\bibitem{GG}  M. Golubitski, V. Guillemin, Stable mappings and their singularities, Springer-Verlag, 1973.

\bibitem{Marcel} M. Guardia, T.M. Seara, M.A. Teixeira, Generic bifurcations of low codimension of planar Filippov Systems, Journal of Differential Equations 250 (2011) 1967--2023.


\bibitem{Livro-Greiner} 	W. Greiner, 
{\it 	Classical Mechanics $-$ Systems of Particles And Hamiltonian Dynamics}. Springer (2003).





\bibitem{Jac-To} A. Jacquemard \& D.J. Tonon,  ``Coupled systems of non-smooth differential equations,'' Bulletin des Sciences Math\'{e}matiques, vol. 136 (2012), 239-255.



\bibitem{JinWuHuang} S. Jin, H. Wu and Z. Huang. {\it A Hybrid Phase-Flow Method for Hamiltonian Systems with Discontinuous Hamiltonians},  SIAM J. Sci. Comput., 31(2), 1303–1321. (19 pages) (2008).




\bibitem{Kousaka}  Kousaka, T.,  Kido, T.,  Ueta, T., Kawakami, H. \& Abe, M. ``Analysis of Border- Collision Bifurcation in a Simple Circuit,'' Proceedings of the International Sym- posium on Circuits and Systems, II-481-II-484, 2000.





\bibitem{Krivan} V. Krivan. {\it On the Gause predator - prey model with a refuge: a fresh look at
	the history},  J. theoret. Biol., 274(1), 67--73, 2011.

\bibitem{Kuznetsov} Y.A. Kuznetsov, S. Rinaldi, A. Gragnani, One-Parameter Bifurcations in Planar Filippov Systems, Int. Journal of Bifurcation and Chaos 13 (2003), 2157--2188.



\bibitem{Leine} Leine, R.  \&  Nijmeijer, H.  ``Dynamics and Bifurcations of Non-Smooth Mechanical Systems,'' Lecture Notes in Applied and Computational Mechanics, vol. 18, Berlin Heidelberg New-York, Springer-Verlag, 2004.

\bibitem{Livro-LlibreMoeckelSimo} J. Llibre, R. Moeckel and C. Simó, 
{\it Central Configurations, Periodic Orbits, and Hamiltonian Systems}. Birkhäuser (2016).



\bibitem{Livro-Luo} A.C.J. Luo, {\it Discontinuous Dynamical Systems},  
Springer (2012).


\bibitem{Piltz} S. H. Piltz, M.A. Porter and P.K. Maini. {\it Prey Switching with a Linear Preference Trade-Off}, SIAM J. Applied Dynamical Systems,
Vol. 13, No. 2, pp. 658–682 (2014).



\bibitem{Rossa} Rossa, F. D. \& Dercole, F. ``Generic and Generalized Boundary Operating Points in Piecewise-Linear (discontinuous) Control Systems,'' In 51st IEEE Conference on Decision and Control, 10-13 Dec. 2012, Maui, HI, USA. Pages:7714-7719. DOI: 10.1109/CDC.2012.6425950.

\bibitem{Simpson} D.J. Simpson.  {\it Bifurcations in Piecewise-Smooth Continuous Systems}, World Scientific Series on Nonlinear Science, Series A,  v. 69 (2010).


\bibitem{Ta73}  F. Takens, Unfoldings of certain singularities of vector fields generalised Hopf bifurcations, J. Differential Equations, 14 (3) (1973), pp. 476--493.

\bibitem{Marco-enciclopedia} M.A. Teixeira,  Perturbation Theory for Non-smooth Systems, Meyers: Encyclopedia of Complexity and Systems Science 152  (2008).


\bibitem{Livro-Hamiltonian} D. Treschev and O. Zubelevich, 
{\it Introduction to the Perturbation Theory of Hamiltonian Systems}. Springer Monographs in Mathematics, Springer-Verlag Berlin Heidelberg (2010).


 \end{thebibliography}
\end{document}